\documentclass[11pt,final]{amsart}
\usepackage[pctex32]{graphics}
\usepackage{amsfonts}
\usepackage{amssymb,amscd,latexsym}
\usepackage{amsmath}
\usepackage{epsfig}
\usepackage{color,colortbl}

\usepackage[dvipsnames]{xcolor}
\usepackage{fancyvrb}   
 
\definecolor{airforceblue}{rgb}{0.36, 0.54, 0.66}
\definecolor{bleudefrance}{rgb}{0.19, 0.55, 0.91}
\definecolor{darkorchid}{rgb}{0.6, 0.2, 0.8}
\definecolor{darkorange}{rgb}{1.0, 0.55, 0.0}
\definecolor{darkspringgreen}{rgb}{0.09, 0.45, 0.27}

\usepackage[arrow,curve,matrix,tips,frame]{xy}

\usepackage[normalem]{ulem}
\usepackage[utf8]{inputenc}
\usepackage[english]{babel} 
\usepackage{enumerate}
\usepackage{hyperref}  
                 \hypersetup{ pdfborder={0 0 0}, 
                              colorlinks=true, 
                              linkcolor=blue, 
                              linktoc=page,
                              pdfauthor={Francesco Russo, Giovanni Staglian{\`o}}, 
                              pdftitle={Explicit rationality of cubic fourfolds}
                            }                         

\usepackage{verbatim}

\newtheorem{thm}{Theorem}

\theoremstyle{definition}

\newcommand{\PP}{{\mathbb{P}}}

\newcommand{\map}{\dasharrow}

\def\p{\mathbb P}

\def\I{\mathcal I}

\newcommand{\Bl}{\operatorname{Bl}}

\newcommand{\Sec}{\operatorname{Sec}}

\def\GG{{\mathbb G}}

\def\PP{{\mathbb P}}

\def\Bl{\operatorname{Bl}}

\def\Rat{\operatorname{Rat}}

\let\phi=\varphi

\newtheorem*{thmn0}{Theorem}

\begin{document}

\title[Explicit rationality of cubic fourfolds]{Explicit rationality of some cubic fourfolds}

\author[F. Russo]{Francesco Russo*}
\address{Dipartimento di Matematica e Informatica, Universit\` a degli Studi di Catania, Viale A. Doria 5, 95125 Catania, Italy}
\email{frusso@dmi.unict.it, giovannistagliano@gmail.com}
\thanks{*Partially  supported  by the PRIN {\it Geometria delle variet\`{a} algebriche} and by the FIR2017  {\it Propriet\`{a} algebriche locali e globali di anelli associati a curve e ipersuperfici} of the University of Catania; the author is a member of the G.N.S.A.G.A. of INDAM}
\author[G. Staglian\` o]{Giovanni Staglian\` o}

\begin{abstract} Recent results of Hassett,  Kuznetsov and others pointed out countably many divisors $C_d$ 
in the open subset of   $\p^{55}=\p(H^0(\mathcal O_{\p^5}(3)))$ parametrizing all cubic 4-folds and lead to the conjecture
that the cubics corresponding to these divisors should be  precisely the rational 
ones. Rationality has been proved by Fano for the first divisor $C_{14}$ and  in \cite{RS1} for
the divisors  $C_{26}$ and  $C_{38}$.
In this note we describe explicit birational maps from a general cubic fourfold in $C_{14}$, in $C_{26}$ and in $C_{38}$
to $\p^4$, providing  concrete geometric realizations of the more abstract constructions in \cite{RS1}, see also \cite{RS3} for a theoretical framework. 
\end{abstract}

\maketitle

\section*{Introduction}
The rationality
of smooth cubic hypersurfaces in $\p^5$ is an open problem  on which a lot of new and interesting contributions and conjectures appeared in the last decades.  The classical work by Fano in \cite{Fano}, correcting some wrong assertions in \cite{Morin}, has been the only known result about the rationality of cubic fourfolds for a long time and, together with a great amount of recent theoretical work on the subject (see for example the survey \cite{Levico}), lead to the expectation that  the very general cubic fourfold should be irrational. More precisely, in the moduli space $\mathcal C$  the locus $\Rat(\mathcal C)$  of rational cubic fourfolds
 is the union of  a countable family of closed subsets $T_i\subseteq \mathcal C$, $i\in\mathbb N$, see \cite[Proposition 2.1]{deFernex2013} and  \cite[Theorem 1]{KontsevichTschinkelInventiones}. 

Hassett defined  in   \cite{Hassett, Has00} (see also \cite{Levico}) via Hodge Theory  infinitely many irreducible 
divisors $\mathcal C_d$  in  $\mathcal C$ and introduced the notion of   {\it admissible values} $d\in\mathbb N$, i.e. those even integers $d>6$  not divisible by 4, by 9 and nor by any odd prime of the form $2+3m$. More recent contributions by Kuznetzsov via derived categories in  \cite{kuz4fold,kuz2}  (see also \cite{AT, Levico}) fortified  the conjecture that
$$\Rat(\mathcal C)=\bigcup_{d \text{ admissible}}\mathcal C_d.$$
The first admissible values are $d=14, 26, 38, 42$ and Fano showed the
rationality of a general cubic fourfold in $\mathcal C_{14}$, see \cite{Fano, BRS}. The  main results of \cite{RS1} are summarised in the following:

\begin{thmn0}\label{intro-rat}
	{\it Every cubic fourfold  in the irreducible divisors $\mathcal C_{26}$ and $\mathcal C_{38}$  is rational.}
	\end{thmn0}

The geometrical definition of $\mathcal C_d$
can be also given as the (closure of the) locus of cubic fourfolds  $X\subset\p^5$ containing an explicit  surface $S_d\subset X$. These surfaces
are obviously not unique and a standard count of parameters shows that in specific examples the previous locus is a divisor (the degree and self-intesection of $S_d$ determine the value $d$ via the formula $d=3\cdot S^2-\deg(S)^2$),
 see \cite{Hassett, Has00} for more details on the Hodge theoretical definition of $\mathcal C_d$ and also \cite{Cd} for recent interesting contributions on  the divisors $\mathcal C_d$.  For example,  the divisor  $\mathcal C_{14}$ can be described either as the closure of the locus
of cubic fourfolds containing a smooth quintic del Pezzo surface or, equivalently, a smooth quartic rational normal scroll, see \cite{Fano, BRS} and also
\cite{Nuer} for other  descriptions with
$12\leq d\leq 44$, $d\neq 42$ or  \cite{Lai} for $d=42$.

Fano proved that the restriction of the linear system of quadrics through a smooth quintic del Pezzo surface, respectively a smooth quartic rational normal scroll, to a general cubic through the surface defines a birational map to $\p^4$, respectively onto a smooth four dimensional quadric hypersurface. Indeed, a general fiber of the map $\p^5\map \p^4$ given by quadrics through a quintic del Pezzo surface is a secant line to it, yielding the birationality of the restriction to $X$ (the other case is similar).
The extension of  this explicit geometrical approach to rationality for other (admissible) values $d$ appeared to be impossible  because there are no other irreducible {\it surfaces with one apparent double point} contained in a cubic fourfold, see \cite{CilRus, BRS}.

In \cite{RS1} we discovered irreducible
surfaces  $S_d\subset\p^5$ admitting a four-dimensional  family of 5-secant conics such that
through a general point of $\p^5$ there passes a unique conic of the family ({\it congruences of
5-secant conics to $S_d$}) for $d$=14, 26 and 38.  From this we deduced the rationality of  a general  cubic in $|H^0(\I_{S_d}(3))|$, showing that  it is    a rational section of the universal family of the congruence of 5-secant conics.  

Here we come back to Fano's method and we propose an explicit realisation  of the previous abstract approach.
Some simplifying hypotheses, based on the known examples in \cite{RS1}, suggest that  rationality might be related to linear systems of  hypersurfaces of degree $3e-1$ having points of multiplicity $e\geq 1$ along a {\it right} surface $S_d$ contained in the general $X\in\mathcal C_d$, see Section \ref{nec}. Clearly the problem is to find the  right $S_d$, prove that the above map is birational and, if the dimension of the linear system is bigger than four, describe the image (which might be highly non trivial). This expectation is motivated by the remark, due to J\' anos Koll\' ar, that if the surface $S_d\subset\p^5$ admits a congruence of $(3e-1)$-secant curves of degree $e\geq 1$ {\it generically transversal} to the cubics through $S_d$ (see Section \ref{nec} for precise definitions), then the above linear systems contract the curves of the congruence. So if the general fiber of the map is a curve of the congruence, then a general cubic through $S_d$ is birational to the image of the associated map. We shall see that, quite surprisingly, this  really occurs  for the first three admissible values $d=14, 26, 38$ and, even more surprisingly, that these linear systems provide by restriction explicit birational maps from a general cubic fourfold in $\mathcal C_d$ for $d=14, 26, 38$ 
to $\p^4$ (and also  to a four dimensional linear section of a $\mathbb G(1,3)$, respectively $\mathbb G(1,4)$, respectively $\mathbb G(1,5)$), a fact which was not known before at least for $d=26, 38$ (see also \cite[Section 5, \S 29]{Kollar}). The theoretical framework explaining the birationality of these maps,  the relations with the theory of congruences to the surfaces $S_d$ and to the associated K3 surfaces has been developed recently in \cite{RS3}.

To analyze   the algebraic and geometric properties of the surfaces $S_d\subset\p^5$ involved as well the particular  linear systems of hypersurfaces of degree $3e-1$  having points of multiplicity $e$ along  the $S_d$'s we used  
Macaulay2 \cite{macaulay2} together with some standard semicontinuity arguments to pass from a particular verification to the general case. 

\medskip

{\bf Acknowledgements}. We  wish to thank J\' anos Koll\' ar for asking about the maps defined by the linear systems of quintics singular along surfaces admitting a congruence of 5-secant conics and for his interest in our subsequent results. 

\begin{table}[htbp]
\centering
\tabcolsep=1.2pt 
\begin{tabular}{cccccccc}
\hline
\rowcolor{gray!5.0}
 $d$  & {Surface $S\subset\PP^5$} & \scriptsize{\begin{tabular}{c} $2$-secant \\ lines \end{tabular}} & \scriptsize{\begin{tabular}{c} $5$-secant \\ conics \end{tabular}} & \scriptsize{\begin{tabular}{c} $8$-secant \\ twisted \\ cubics \end{tabular}}  & $h^0(\mathcal{I}_{S/\PP^5}(3))$ & $h^0(N_{S/\PP^5})$ & $h^0(N_{S/X})$ \\
\hline \hline 
\rowcolor{blue!5.0}
${14}$ & \scriptsize{\begin{tabular}{c} Smooth del Pezzo surface of degree $5$ \end{tabular}} & $1$  & $0$ &  $0$ & $25$ & $35$ & $5$  \\
\hline
\rowcolor{red!5.0}
${14}$ & \scriptsize{\begin{tabular}{c} Rational normal scroll of degree $4$ \end{tabular}} & $1$  & $0$ &  $0$ & $28$ & $29$ & $2$   \\
\hline
\rowcolor{blue!5.0}
${14}$ & \scriptsize{\begin{tabular}{c} Isomorphic projection of a smooth surface \\in $\p^6$ of degree $8$ and sectional genus $3$,\\ obtained as the image of $\p^2$ via the linear \\system of quartic curves with $8$ general \\ base points \end{tabular}} & $7$  & $1$ &  $0$ & $13$ & $49$ & $7$ \\
\hline
\rowcolor{blue!5.0}
${26}$ & \scriptsize{\begin{tabular}{c} Rational scroll of degree $7$ with $3$ nodes  \end{tabular}} & $7$  & $1$ &  $0$ & $13$ & $44$ & $2$ \\
\hline
\rowcolor{blue!5.0}
${38}$ & \scriptsize{\begin{tabular}{c} Smooth surface of degree $10$ and sectional \\ genus $6$, obtained as the image of $\PP^2$ via \\ the linear system  of curves of degree $10$\\  with $10$ general triple points \end{tabular}} & $7$  & $1$ &  $0$ & $10$ & $47$ & $2$ \\
\hline
\rowcolor{red!5.0}
${26}$ & \scriptsize{\begin{tabular}{c} Projection of a smooth del Pezzo surface \\ of degree $7$ in $\PP^7$  from a line intersecting \\ the secant variety in one general point \end{tabular}} & $5$  & $1$ &  $0$  & $14$ & $42$ & $1$ \\
\hline
\rowcolor{red!5.0}
${38}$ & \scriptsize{\begin{tabular}{c} Rational scroll of degree $8$ with $6$ nodes  \end{tabular}} & $9$  & $4$ &  $1$ & $10$ & $47$ & $2$   \\
\hline
\end{tabular}
 \caption{Surfaces $S\subset\PP^5$ contained in a cubic fourfold $[X]\in\mathcal C_d$ and admitting a congruence of $(3e-1)$-secant rational normal curves of degree $e\leq3$.} 
\label{TabCongruenze} 
\end{table}

\begin{table}[htbp]
\centering
\tabcolsep=5.8pt 
\begin{tabular}{ccccccccc}
\hline
\rowcolor{gray!5.0}
 $d$  & $e$ & Multidegree & $Y^4$ & $\delta$  & $\deg(\mathfrak{B})$ & $g(\mathfrak{B})$ & $\deg(\mathfrak{B}_{\mathrm{red}})$ & $g(\mathfrak{B}_{\mathrm{red}})$\\
\hline \hline 
\rowcolor{blue!5.0}
${14}$ & $1$ & $3, 6, 7, 4, 1$  & $\PP^4$ &  $4$ & $9$ & $8$ & $9$ & $8$  \\
\hline
\rowcolor{red!5.0}
${14}$ & $1$ & $3, 6, 8, 6, 2$   & $\GG(1,3)\subset\PP^5$ &  $3$ & $10$ & $7$ & $10$ & $7$   \\
\hline  
\rowcolor{blue!5.0}
${14}$ & $2$ & $3, 15, 19, 9, 1$  & $\PP^4$ &  $9$ & $52$ & $256$ & $32$ & $106$ \\
\hline
\rowcolor{blue!5.0}
${26}$ & $2$ & $3, 15, 20, 9, 1$  & $\PP^4$ &  $9$ & $51$ & $246$ & $31$ & $100$ \\
\hline
\rowcolor{blue!5.0}
${38}$ & $2$ & $3, 15, 27, 9, 1$  & $\PP^4$ &  $9$ & $42$ & $165$ & $18$ & $39$ \\
\hline
\rowcolor{red!5.0}
${26}$ & $2$ & $3, 15, 31, 25, 5$  & $\GG(1,4)\cap\PP^7\subset\PP^7$ &  $5$  & $77$ & $212$ & $43$ & $73$ \\
\hline
\rowcolor{red!5.0}
${38}$ & $3$ & $3, 24, 80, 70, 14$  & $\GG(1,5)\cap\PP^{10}\subset\PP^{10}$ &  $5$ & $204$ & $633$ & $94$ & $144$  \\
\hline
\end{tabular}
 \caption{Birational maps from a  cubic fourfold $[X]\in\mathcal C_d$ to a fourfold $Y^4$
 defined by the restrictions to $X$ of the linear systems 
 $|H^0(\mathcal I_S^e(3e-1))|$, where $S\subset\PP^5$ 
 are the surfaces in Table~\ref{TabCongruenze} admitting a congruence of $(3e-1)$-secant rational normal curves of degree $e\leq3$.
 Here, $\mathfrak{B}$ denotes the base locus of the inverse map, which is a $2$-dimensional scheme; $g$ 
 stands for the sectional arithmetic genus; $\delta$ is the degree of the forms defining the inverse map.} 
\label{TabCongruenze2} 
\end{table}

\section{Explicit  rationality  via linear systems of hypersurfaces of degree $3e-1$ having points of multiplicity $e$ along a {\it right} surface}\label{nec}

Let us recall the following definitions introduced in \cite[Section 1]{RS1}. Let $\mathcal H$ be 
an irreducible proper family of  (rational or of fixed arithmetic genus) curves of degree $e$ in $\p^5$ whose general element is irreducible.
We have a diagram
$$\xymatrix{
 \mathcal D\ar[dr]^{\psi}\ar[d]^{\pi}&\\
\mathcal H&\p^5}
$$
where $\pi:\mathcal D\to \mathcal H$
is the  universal family over $\mathcal H$ and where  $\psi:\mathcal D\to\p^5$ is  the tautological morphism. 
Suppose moreover that $\psi$ is birational and that a general member $[C]\in \mathcal H$ is ($re-1$)-secant to an irreducible surface $S\subset\p^5$,
that is $C\cap S$ is a length $r e-1$ scheme, $r\in\mathbb N$. We shall call such a family $\mathcal H$ a {\it congruence of }  ($re-1$)-{\it secant curves of degree $e$ to $S$}.  Let us remark that necessarily $\dim(\mathcal H)=4$.
\medskip

An irreducible  hypersurface $X\in|H^0(\mathcal I_{S}(r))|$ is said to be {\it transversal to the congruence $\mathcal H$} if  the unique curve of the congruence passing through a general point  $p\in X$ is not contained in $X$. A crucial result is the following.

\begin{thm}\label{criterion} {\rm \cite[Theorem 1]{RS1}} Let $S\subset\p^5$ be a surface admitting a congruence of {\rm(}$re-1${\rm)}-secant curves of degree $e$ parametrized by $\mathcal H$.
If  $X\in|H^0(\mathcal I_{S}(r))|$ is an irreducible hypersurface transversal to $\mathcal H$,  then $X$ is birational to $\mathcal H$.

If the map   
$\Phi=\Phi_{|H^0(\mathcal I_{S}(r))|}:\p^5\map \p(H^0(\mathcal I_{S}(r)))$
is birational onto its image, then a general  hypersurface  $X\in|H^0(\mathcal I_{S}(r))|$  is birational to $\mathcal H$.

Moreover, under the previous hypothesis on $\Phi$, if  a general  element in $|H^0(\mathcal I_{S}(r))|$ is smooth, then every $X\in |H^0(\mathcal I_{S}(r))|$ with at worst rational singularities is birational to $\mathcal H$.
\end{thm} 

Thus, when a surface $S\subset\p^5$ admits a conguence of $(3e-1)$-secant curves of degree $e\geq 1$ which is transversal to a general
 cubic fourfold through it, such a  cubic fourfold is birational to $\mathcal H$ via $\pi\circ\psi^{-1}_{|X}$, see the proof of the above result.
Obviously,  the difficult key point is to find a congruence as above with $\mathcal H$ rational (or irrational, depending on the application).

Since $\psi:\mathcal D\to\p^5$ is birational, we also have a rational map $$\varphi=\pi\circ\psi^{-1}:\p^5\map\mathcal H,$$ whose general fiber through $p\in\p^5$, $F=\overline{\varphi^{-1}(\varphi(p))}$, is the unique curve of the congruence passing through $p$. On the contrary, if there exists a map $\varphi:\p^5\map Y\subseteq\p^N$ with $Y$ a four dimensional variety, whose general fiber $F$ is an irreducible curve of degree $e$ which is $(3e-1)$-secant to a surface $S\subset\p^5$, then we have found a congruence for $S$, a birational realization of the variety $\mathcal H$ and a concrete representation of the abstract map $\pi\circ\psi^{-1}$.

It is natural to ask what linear systems on $\p^5$ can give  maps $\varphi:\p^5\map Y$ as above. Since  
a linear system in $|H^0(\mathcal I_S^e(3e-1))|$ contracts the fibers of $\varphi$, these (complete) linear systems appear as natural potential candidates, as remarked by J\' anos Koll\' ar. 
A posteriori we shall see that, quite surprisingly, this  really occurs with $Y=\p^4$ (or with $Y$ a linear section of a Grassmannian of lines) for the first three admissible values $d=14, 26, 38$ and, even more surprisingly, that these linear systems provide the explicit rationality of a general cubic fourfold in $\mathcal C_d$ for $d=14, 26, 38$. 
\medskip

Let $E\subset \Bl_S\p^5$ be the exceptional divisor of the blow-up of $\p^5$ along $S$, let $H\subset 
\Bl_S\p^5$ be the pull back
of a hyperplane in $\p^5$, let $F'$ be the strict transform of $F$  and let $\tilde\varphi:\Bl_S\p^5\map \mathcal H$ be the rational map induced by $\varphi$.
Then  $E\cdot F'=3e-1$ by hypothesis and $(3H-E)\cdot F'=1$.
The last condition  translates both  that a general cubic through  $S$ is mapped birationally onto $\mathcal H$ by $\varphi$, both the fact that the linear system
$|H^0(\mathcal I_S(3))|$ sends a general $F$ into a line contained in  the image of the corresponding map $\psi:\p^5\map \p(H^0(\mathcal I_S(3)))$ and, last but not least, also that the congruence is transversal to a general cubic through $S$ (see \cite{RS1}
for a systematic use of these key remarks).
\medskip

For $e=1$ one should consider linear systems of quadric hypersurfaces  through $S$; for $e=2$  quintics having double points along $S$; for $e=3$ hypersurfaces of degree 8 having triple points along $S$ and so on. 

For $e=1$ we have a unique secant line to $S$ passing through a general point of $\p^5$, which is a very strong restriction. Indeed, such a $S$ is a so called {\it surface with one apparent double point}. These surfaces  are completely classified in \cite{CilRus} and those contained in a  cubic fourfold are only quintic del Pezzo's and smooth quartic rational normal scrolls. Cubic fourfolds through these surfaces describe the divisor $\mathcal C_{14}$ as it was firstly remarked by Fano in \cite{Fano} (see also \cite[Theorem 3.7]{BRS} for a modern account of Fano's  original arguments using  deformations of quartic scrolls). 

Let $D\subset \p^5$ be an arbitrary  smooth quintic del Pezzo surface. Then $|H^0(\mathcal I_D(2))|=\p^4$ and this linear system determines a dominant rational map $\varphi:\p^5\map \p^4$, whose general fiber $F$ is a secant line to $D$. Then the restriction of $\varphi$ to a  cubic fourfold through $D$ yields a birational map
$\psi:X\map \p^4$ and  hence the rationality of a general $X\in\mathcal C_{14}$, as firstly remarked by Fano in \cite{Fano}.

Let $T\subset\p^4$ be a smooth quartic rational normal scroll. Then $|H^0(\mathcal I_T(2))|=\p^5$ and this linear system determines a dominant rational map $\varphi:\p^5\map Q\subset \p^5$, whose general fiber $F$ is a secant line to $T$ and with $Q$ a smooth quadric hypersurface. Then the restriction of $\varphi$ to a  cubic fourfold through $T$ yields a birational map
$\psi:X\map Q$ and another proof of the rationality of a general $X\in\mathcal C_{14}$, see \cite{Fano}. 
\medskip

We shall mainly consider the surfaces $S_d\subset\p^5$ admitting a congruence of 5-secant conics parametrised by a rational variety studied  in \cite{RS1} (but also other new examples) to  determine  explicitly the rationality of a general  $X\in\mathcal C_d$ for $d=14, 26, 38$ with $e=2$ (or also for other values $e\geq 3$). To this aim we shall summarise some well known facts in the next subsection.

\subsection{Linear systems of quintics with double points along a general $S_d\in \mathcal S_d$}\label{construction} Let $\mathcal S_d$ be an irreducible component of the Hilbert  scheme of surfaces in $\p^5$ with a fixed Hilbert polynomial $p(t)$ and such that 
$$\mathcal C_d=\overline{\{[X]\in\mathcal C\text{ for which }\exists\,  [S_d]\in\mathcal S_d\,:\, S_d\subset X\}}.$$

One can verify explicitly the previous equality by comparing the Hodge theoretic definition on the left with the geometrical description on the right.
A modern count of parameters  usually shows that the right side is at least a divisor in $\mathcal C$ so that equality holds because $\mathcal C_d$ is an irreducible divisor if not empty. 
The hard problem is to compute the dimension
of the family of $S_d$'s contained in a fixed (general)  $X$ belonging to the set on the right side above, see \cite{Nuer, RS1} for some efficient computational arguments based on semicontinuity.

For every $a,b\in \mathbb N$ the functions $h^0(\mathcal I_{S_d}^b(a))$ are upper semicontinuous on $\mathcal S_d$. In particular there exists an open non empty subsets $U\subseteq \mathcal S_d$ on which $h^0(\mathcal I_{S_d}^b(a))$ attains a minimum value $m=m(a,b)$.

We shall be mainly interested in the case $a=5$ and $b=2$, that is the computation of  the dimension of the linear system $|H^0(\mathcal I_{S_d}^2(5))|$ for $S_d\in\mathcal S_d$ general. To this aim we  consider the exact sequence
\begin{equation}\label{conorm}
0\to \mathcal I^2_{S_{d}}(5)\to \mathcal I_{S_{d}}(5)\to N^*_{S_{d}/\p^5}(5)\to 0.
\end{equation}

Suppose that we know $h^0(N^*_{S_{d}/\p^5}(5))=y$ 
and $h^0(\mathcal I_{S_d}(5))=x$ for the general $S_d\in\mathcal S_d$ via standard exact sequences (or also computationally) or for some geometrical property of the surfaces. From \eqref{conorm} we deduce $h^0(\mathcal I^2_{S_d}(5))\geq x-y$ for a general $S_d$. By the upper semicontinuity of $h^0(\mathcal I^2_{S_d}(5))$ it will be sufficient to find a surface $S\in \mathcal S_d$ with $h^0(\mathcal I^2_S(5))=x-y$ to deduce that the same holds
for a general $S_d\in \mathcal S_d$.

If $\pi:\chi_d\to \mathcal S_d$ is the universal family and if a general $[S_d]\in\mathcal S_d$ is  smooth,
let $V\subseteq\mathcal S_d$ be the non empty open set of points $[S]\in\mathcal S_d$ such that $S\subset\p^5$ is a smooth surface. Then $\pi^{-1}(V)\to V$ is a smooth morphism and   the function   $h^0(N^*_{S/\p^5}(a))$ is upper semicontinuos on $V$ for every $a\in\mathbb N$. In particular, if there exists $[S]\in V$ such that   $h^0(N^*_{S/\p^5}(a))=z$, then, for a general $[S_d]\in \mathcal \mathcal S_d$, we have  $h^0(N^*_{S_d/\p^5}(a))\leq z$ and hence $h^0(\mathcal I_{S_d}^2(a))\geq m(a,2)-z$ .  If moreover $h^0(\mathcal I_{S'}^2(a))=m(a,2)-z$ for a $[S']\in\mathcal S_d$, then the same holds for a general $[S_d]\in\mathcal S_d$. These standard and well known remarks will be useful in our analysis of the examples in the next sections, where we shall also deal with the case $e=3$ and the linear systems $|H^0(\mathcal I_{S_d}^3(8))|$ for $d=14,38$.

\subsection{Computations via Macaulay2} 
To study surfaces in $\p^5$ admitting congruences of $(3e-1)$-secant curves
of degree $e$, the rational maps given by hypersurfaces of degree $3e-1$ having points of multiplicity $e$ along these surfaces
and also the lines contained in the images
of $\p^5$ via the linear system of cubics through these surfaces we mostly  used  
Macaulay2 \cite{macaulay2}. 

Our proofs of various claims 
exploit the fact that 
the irreducible components $\mathcal S_d$ 
of the Hilbert schemes considered here 
are unirational.
Therefore, by introducing a finite number of free parameters, 
one can explicitly construct  the generic surface in $\mathcal S_d$ in function of the specified parameters. 
Adding more parameters one can also take the generic point of $\p^5$,
and then one 
can for instance
compute the generic fiber 
of the map
defined by the cubics through 
the generic $[S_d]\in\mathcal S_d$, which will depend on all these parameters. 
In principle, there are no theoretical limitations to perform this computation,
but in practice this is far beyond what computers can do today. Anyway,
the answer we get is equivalent to the one obtained 
on the original field via a generic specialization 
of the parameters and, above all, 
the generic specialization commutes with this type of computation. 
So, using a common computer one can  get an experimental proof  
that a certain property holds or not for the generic $[S_d]$.
In the affirmative case, one then applies some
semicontinuity arguments to get a rigorous proof.

\section{Explicit birational maps to $\p^4$  via linear system of quintics with double points
for cubics in $\mathcal C_{d}$,  $d=14,26, 38$.}\label{explicit}

Let  $S_{14}\subset\p^5$ be an isomorphic projection of an  octic  smooth
surface $S\subset\p^6$ of sectional genus 3 in $\p^6$, obtained as the image of $\p^2$ via the linear system of quartic
curves with 8  general base points. Let $\mathcal S_{14}$ be the irreducible component  of the Hilbert scheme parametrizing  surfaces $S_{14}\subset\p^5$ as above. 
In \cite[Theorem 2]{RS1} we verified that a general $X\in\mathcal C_{14}$ contains a surface $S_{14}$
and that each $S_{14}$ admits a congruence of 5-secant conics (birationally) parametrized by its symmetric product and transversal to $X$.

\begin{thm}\label{famconics14} For a  general surface $S_{14}\subset\p^5$ as above we have $|H^0(\mathcal I_{S_{14}}^2(5))|=\p^4$ and this linear system determines a rational map $\varphi:\p^5\map \p^4$ whose general fiber is a  5-secant conic to $S_{14}$. In particular,  the restriction of $\varphi$ to a general cubic through $S_{14}$ is birational.
\end{thm}
\begin{proof}
A general $S_{14}\in\mathcal S_{14}$  is $k$-normal for every $k\geq 2$, see \cite[Example 3.8]{AR}, yielding $h^0(\mathcal I_{S_{14}}(5))=141$. For a particular smooth $S\in\mathcal S_{14}$ we verified that  $h^0(N^*_{S/\p^5}(5))=136$ so that for a general $S_{14}\in\mathcal S_{14}$  we have $h^0(N^*_{S_{14}/\p^5}(5))\leq 136$. By \eqref{conorm} we have   $h^0(\mathcal I^2_{S_{14}}(5))\geq 5$ for a general $S_{14}\in\mathcal S_{14}$. Since in the example we studied   $h^0(\mathcal I^2_{S}(5))=5$, the same holds for a general $S_{14}\in\mathcal S_{14}$, see Subsection \ref{construction}.

Let $\varphi: \p^5\map\p^4$ be the rational  map associated to $|H^0(\mathcal I^2_{S_{14}}(5))|=\p^4$ with $S_{14}$ general. 
We verified that the closure of a general fiber of $\varphi$ is a 5-secant conic to $S_{14}$,
concluding the proof.
\end{proof}

Let $S_{26}\subset\p^5$ be a  rational septimic
scroll with three nodes recently considered by Farkas and Verra in \cite{FV}, where they also proved that a general $X\in\mathcal C_{26}$ contains a surface of this kind. Also these surfaces admit a congruence of
5-secant conics transversal to $X$ and parametrized by a rational variety, see \cite[Remark~6]{RS1}.

\begin{thm}\label{famconics26} For a general  surface $S_{26}\subset\p^5$ as above we have $|H^0(\mathcal I_{S_{26}}^2(5))|=\p^4$ and this linear system determines a rational map $\varphi:\p^5\map \p^4$ whose general fiber is a  5-secant conic to $S_{26}$. In particular,  the restriction of $\varphi$ to a general cubic through $S_{26}$ is birational.
\end{thm}
\begin{proof} For a general $S_{26}\in\mathcal S_{26}$ as above,  we have $h^0(\mathcal I_{S_{26}}(5))=144$, $h^0(N^*_{S_{26}/\p^5}(5))=139$ and in an explicit example of $S\in\mathcal S_{26}$ we verified  that $h^0(\mathcal I_S^2(5))=5$. Thus  $|H^0(\mathcal I_{S_{26}}^2(5))|=\p^4$ for a general $S_{26}$, see Subsection \ref{construction}. Let $\varphi: \p^5\map\p^4$ be the rational  map associated to $|H^0(\mathcal I^2_{S_{26}}(5))|=\p^4$ with $S_{26}$ general. 
We verified that  a general  fiber of the corresponding $\varphi$ is a 5-secant conic to $S_{26}$, concluding the proof.
\end{proof}

Let $S_{38}\subset\p^5$ be a general  degree 10 smooth surface of sectional genus 6 obtained as the image
of $\p^2$ by the linear system of plane curves of degree 10 having 10 fixed triple points. As shown by Nuer in \cite{Nuer},  these surfaces are contained in a general $[X]\in\mathcal C_{38}$.   In \cite[Theorem 4]{RS1}
we proved that a general $S_{38}$ admits a congruence of 5-secant conics transversal to $X$ and parametrised by a rational variety.

\begin{thm}\label{famconics38} For a  surface $S_{38}\subset\p^5$ as above we have $|H^0(\mathcal I_{S_{38}}^2(5))|=\p^4$ and this linear system defines a rational map $\varphi:\p^5\map \p^4$ whose general fiber is a  5-secant conic to $S_{38}$. In particular,  the restriction of $\varphi$ to a general cubic through $S_{38}$ is birational.
\end{thm}
\begin{proof}
A  general $S_{38}\in\mathcal S_{38}$ has ideal generated by 10 cubic forms and is thus 5-normal by \cite[Proposition 1]{bertram-ein-lazarsfeld}, yielding $h^0(\mathcal I_{S_{38}}(5))=126$ for a general $S_{38}$ (a fact which can also be verified by a direct computation). For a particular smooth $S\in\mathcal S_{38}$ we verified that  $h^0(N^*_{S/\p^5}(5))=121$ so that for a general $S_{38}\in\mathcal S_{38}$  we have $h^0(N^*_{S_{38}/\p^5}(5))\leq 121$. From \eqref{conorm}   we deduce $h^0(\mathcal I^2_{S_{38}}(5))\geq 5$ for a general $S_{38}\in\mathcal S_{38}$. Since in the previous  explicit  example we also  have $h^0(\mathcal I^2_{S}(5))=5$,  the same holds  for a general $S_{38}\in\mathcal S_{38}$, see Subsection \ref{construction}. We verified that a general fiber of the corresponding rational map $\varphi$ is a 5-secant conic to $S_{38}$, concluding the proof.
\end{proof}

\section{Explicit birational maps to linear sections of $\mathbb G(1,3+k)$ for cubics in $\mathcal C_{14+12k}$ for $k = 1,2$}\label{Gr}

In this section we  analyse some examples and look at  them as suitable generalisation of those considered by Fano. A smooth quintic del Pezzo surface $D\subset\p^5$ can be realized as a divisor of type $(1,2)$ on the Segre 3-fold $Z=\p^1\times \p^2\subset\p^5$ while a smooth quartic rational normal scroll $T\subset\p^5$ can be realized (also) as a divisor of type $(2,1)$. Moreover, a general cubic through $D$ will cut $Z$ along $D$ and a smooth divisor of type $(2,1)$, that is a smooth rational normal scroll $T$ (and viceversa).

One might wonder if  something similar happens for the next admissible values $d=26$ and $d=38$  or if, at least, also in these cases there exist  surfaces $S_d$ giving explicit birational maps to four dimensional (smooth) linear sections of $\mathbb G(1,3+k)$,  $d=14+12k$. We shall see that very surprisingly this is the case although the linkage phenomenon described above appears again only for $d=38$. 
\medskip

Let $S\subset\p^6$ be a septimic surface with a node, which is the projection of a smooth del Pezzo
surface of degree seven in $\p^7$ from a general point on its secant variety. Let $S'_{26}\subset\p^5$ be the projection of $S$ from a general point outside the secant variety $\Sec(S)\subset\p^6$. These surfaces admit a congruence of
5-secant conics parametrised by a rational variety and a general cubic in $\mathcal C_{26}$ contains such a surface, see \cite[Theorem 4]{RS1}.

\begin{thm}\label{f26} For a general  surface $S'_{26}\subset\p^5$ as above we have $|H^0(\mathcal I_{S'_{26}}^2(5))|=\p^7$ and this linear system determines a rational map $\varphi:\p^5\map Y^4\subset \p^7$ with $Y^4$ a smooth linear section of $\mathbb G(1,4)\subset\p^9$. A  general fiber of $\varphi$ is a  5-secant conic to $S'_{26}$ and  the restriction of $\varphi$ to a general cubic through $S'_{26}$ is birational. 
\end{thm}
\begin{proof} For a general $S'_{26}\in\mathcal S'_{26}$ as above
we verified  that $h^0(\mathcal I_{S'_{26}}^2(5))=8$ and that the closure of the image of  $\varphi:\p^5\map\p^7$ is  a smooth four dimensional linear section
$Y^4$ of $\mathbb G(1,4)$. Moreover,  a general  fiber of the corresponding map $\varphi$ is  a 5-secant conic to $S'_{26}$, concluding the proof.
\end{proof}

 The two surfaces $S_{26}$ and $S'_{26}$  are not linked in a variety of dimension three via cubics because the sum of their degrees is 14, which  is not divisible by 3.
\medskip
  
Studying  the rational map $\varphi:\p^5\map\p^4$ treated in Theorem \ref{famconics38} we realized that its base locus contains an irreducible component of dimension three  $B\subset\p^5$ of degree 6 and sectional genus 3. This variety 
 has 7 singular points and it has homogeneous ideal generated by four cubics. So $B\subset\p^5$ is a degeneration of the so called {\it Bordiga scroll},
 which is  a threefold given by the maximal minors of a general $3\times 4$ matrix of linear forms on $\p^5$. The variety $B$   contains the surface $S_{38}\subset\p^5$ and a  general cubic through 
 $S_{38}$ cuts $B$ along $S_{38}$ and an octic rational scroll $S'_{38}\subset\p^5$  with  6 nodes belonging to  the singular locus of  $B$. The scroll $S'_{38}$ is a projection of a smooth octic rational normal scroll $S\subset\p^9$ from a special $\p^3$ cutting the secant variety to $S$ in six points.

  As far as we know this octic rational scroll with six nodes has not been constructed before and, in this context, it is the right generalization of the smooth quartic rational normal scroll considered by Fano. Moreover, it is   remarkable also because it does not come from  the diagonal construction via the associated $K3$ surfaces as for the Farkas-Verra and Lai scrolls (see \cite{FV,Lai} also for more details on this construction). Let us now describe some geometrical properties of the octic rational scroll $S'_{38}\subset\p^5$.
 
 \begin{thm}\label{f38} A general $[X]\in\mathcal C_{38}$ contains an octic rational scroll $S'_{38}\subset\p^5$ with 6 nodes. Moreover, for a general  surface $S'_{38}\subset\p^5$ we have $|H^0(\mathcal I_{S'_{38}}^3(8))|=\p^{10}$ and this linear system determines a rational map $\varphi:\p^5\map Y^4\subset \p^{10}$ with $Y^4$ a  linear section of $\mathbb G(1,5)\subset\p^{14}$. The  general fiber of $\varphi$ is an  8-secant twisted cubic to $S'_{38}$ and   the restriction of $\varphi$ to a general cubic through $S'_{38}$ is birational.   In particular, an octic rational scroll $S'_{38}\subset\p^5$ admits a congruence of 8-secant cubics.
\end{thm}
\begin{proof} A general octic scroll $S'_{38}\subset\p^5$ depends on 47 parameters and it has homogeneous ideal generated by 10 cubics forms.  In an explicit example $S\in\mathcal S'_{38}$ we verified that $h^0(N_{S/\p^5})=47$ and that $S$ is contained in smooth cubic hypersurfaces. Therefore $\mathcal S'_{38}$ is generically smooth of dimension 47 and the natural incidence correspondence in $\mathcal S'_{38}\times\mathcal C$ above the open subset of $\mathcal S'_{38}$ where $h^0(\mathcal I_S(3))=10$ has an irreducible component $\widetilde{\mathcal C_{38}}$ of dimension 56. Since $(S'_{38})^2=34$ to prove that a general $[X]\in\mathcal C_{38}$ contains such a surface, we verified  in an explicit general example that $h^0(N_{S'_{38}/X})=2$.

For a general $S'_{38}\in\mathcal S'_{38}$   we have $h^0(\mathcal I_{S'_{38}}^3(8))=11$ and the closure of the image of the associated rational map $\varphi: \p^5\map \p^{10}$ is  a  linear section $Y^4\subset\p^{10}$ of $\mathbb G(1,5)\subset\p^{14}$.  
We verified that  a general fiber of  $\varphi$ is an 8-secant twisted cubic to $S'_{38}$, concluding the proof.
\end{proof}

\section{Computations}

The aim of this section is to show how one can ascertain
the contents of Theorems~\ref{famconics14}, \ref{famconics26}, \ref{famconics38}, and similar results
in specific examples using the computer algebra system  {Macaulay2}~\cite{macaulay2}.

We begin to observe that 
given the defining homogeneous ideal of a subvariety
$X\subset\mathbb{P}^n$, the computation 
of a basis for the linear system 
$|H^0(\mathcal I_{X}^e(d))|$ of hypersurfaces of degree $d$
 with points of multiplicity at least $e$ along $X$ can be  perfomed using pure linear algebra.
 This 
  approach is implemented in
 the {Macaulay2} package \emph{Cremona} (see \cite{packageCremona}), 
which turns out to be effective for small values of $d$ and $e$.
In practice, 
in any {Macaulay2} session with the \emph{Cremona} package loaded, if \texttt{I}
is a variable containing the ideal of $X$,
we get a rational map defined by
a basis of $|H^0(\mathcal I_{X}^e(d))|$ 
by the command\footnote{For all the examples treated in this paper, 
the linear system of hypersurfaces of degree $d$ with
points of multiplicity at least $e$ along $X\subset\PP^n$
coincides with the homogeneous component of degree $d$
of the saturation 
with respect to the irrelevant ideal of $\PP^n$ 
of the $e$-power of the homogeneous ideal of $X$.
So one can also compute it using the code:
\texttt{gens image basis(d,saturate(I\textasciicircum{}e))}.}
\texttt{rationalMap(I,d,e)}.

Now we consider a specific example related to Theorem~\ref{famconics38}. 
In the following  code, we 
produce a pair $(f,\varphi)$ of rational maps:
$f:\PP^2\dashrightarrow\PP^5$ is a birational parameterization of a 
smooth surface $S=S_{38}\subset \PP^5$ of degree $10$ and sectional genus $6$ as in Theorem~\ref{famconics38},
and $\varphi:\PP^5\dashrightarrow\PP^9$
is a rational map defined by a basis of cubic hypersurfaces containing 
$S$ (see also Section 5 of \cite{RS1}). Here we work over the finite field $\mathbb{F}_{10000019}$ for speed reasons. 
{\footnotesize
\begin{Verbatim}[commandchars=&!$]
Macaulay2, version 1.14
with packages: &colore!airforceblue$!ConwayPolynomials$, &colore!airforceblue$!Elimination$, &colore!airforceblue$!IntegralClosure$, &colore!airforceblue$!InverseSystems$, 
               &colore!airforceblue$!LLLBases$, &colore!airforceblue$!PrimaryDecomposition$, &colore!airforceblue$!ReesAlgebra$, &colore!airforceblue$!TangentCone$, &colore!airforceblue$!Truncations$
&colore!darkorange$!i1 :$ &colore!airforceblue$!needsPackage$ "&colore!airforceblue$!Cremona$"; 
&colore!darkorange$!i2 :$ f = &colore!airforceblue$!rationalMap$(&colore!darkspringgreen$!ZZ$/10000019[&colore!airforceblue$!vars$(0..2)],{10,0,0,10});
o2 : RationalMap (rational map from PP^2 to PP^5)
&colore!darkorange$!i3 :$ S = &colore!airforceblue$!image$ f;
&colore!darkorange$!i4 :$ phi = &colore!airforceblue$!rationalMap$ S;
o4 : RationalMap (cubic rational map from PP^5 to PP^9)
\end{Verbatim}
} \noindent 
We now compute the rational map $\psi$ defined by the linear system of 
quintic hypersurfaces of $\PP^5$ which are singular along $S$. 
From the information obtained by its projective degrees 
we deduce
that $\psi$ is a dominant rational map onto $\PP^4$
with generic fibre of dimension $1$ and degree $2$ and with
base locus of dimension $3$ and degree $5^2-19=6$.
{\footnotesize
\begin{Verbatim}[commandchars=&\[\]]
&colore[darkorange][i5 :] &colore[darkorchid][time] psi = &colore[airforceblue][rationalMap](S,5,2);
     &colore[Sepia][-- used 9.07309 seconds]
o5 : RationalMap (rational map from PP^5 to PP^4)
&colore[darkorange][i6 :] &colore[airforceblue][projectiveDegrees] psi
o6 = {1, 5, 19, 13, 2, 0}
\end{Verbatim}
} \noindent 
Next we compute a special random fibre $F$ of the map $\psi$.
{\footnotesize
\begin{Verbatim}[commandchars=&\[\]]
&colore[darkorange][i7 :] p = &colore[airforceblue][point source] psi;  &colore[Sepia][-- a random point on P^5]
&colore[darkorange][i8 :] F = psi^*(psi(p));             
\end{Verbatim}
} \noindent 
It easy to verify directly that $F$ is an irreducible $5$-secant conic to $S$ passing through $p$. 
One can also see that 
$F$ 
coincides with the pull-back $\overline{\varphi^{-1}(L)}$ of the unique line $L\subset \overline{\varphi(\PP^5)}\subset\PP^9$ 
 passing through $\varphi(p)$ that is not the image 
 of a secant line to $S$ passing through $p$ (see \cite{RS1} for details on this computation).
Finally, the following lines of code tell us
that the restriction of $\psi$ to a random cubic fourfold containing $S$ is a birational map 
whose inverse map is defined by forms of degree $9$ and has base locus scheme
of dimension $2$ and degree $9^2 - 27 = 54$.
{\footnotesize
\begin{Verbatim}[commandchars=&\[\]]
&colore[darkorange][i9 :] psi' = psi|&colore[airforceblue][sum](S_*,i->&colore[airforceblue][random](ZZ/10000019)*i);
o9 : RationalMap (rational map from hypersurface in PP^5 to PP^4)
&colore[darkorange][i10 :] &colore[airforceblue][projectiveDegrees] psi'
o10 = {3, 15, 27, 9, 1}
\end{Verbatim}
} \noindent 

For the convenience of the reader, we have included in a 
Macaulay2 package
(named \emph{ExplicitRationality}
and provided as an ancillary file to our arXiv submission)
the examples listed in Table~\ref{TabCongruenze2}
of birational maps between 
cubic fourfolds and other rational fourfolds. 
The two examples with $d=38$ can be obtained as follows:
{\footnotesize
\begin{Verbatim}[commandchars=&\[\]]
&colore[darkorange][i11 :] &colore[airforceblue][needsPackage] "&colore[airforceblue][ExplicitRationality]"; 
&colore[darkorange][i12 :] &colore[darkorchid][time] &colore[airforceblue][example38]();
      &colore[Sepia][-- used 1.2428 seconds]
\end{Verbatim}
} \noindent 
The above command produces the following $6$ rational maps:
\begin{enumerate}
 \item a parameterization $f:\PP^2\dashrightarrow  \PP^5$ as that obtained above  of a
 surface $S=S_{38}\subset \PP^5$ of degree $10$ and sectional genus $6$;
 \item the rational map $\psi:\PP^5\dashrightarrow\PP^4$ defined by the quintic hypersurfaces with double points along $S$;
 \item the restriction of $\psi$ to a cubic fourfold $X$ containing $S$, which is a birational map;
 \item the linear projection of a smooth scroll surface of degree $8$ in $\PP^9$ 
       from a linear $3$-dimensional subspace intersecting the secant variety of the scroll in  $6$ points,
       so that the image is a scroll surface $T\subset\PP^5$ of degree $8$ with $6$ nodes;
       moreover, we have the relation: $T\cup S = X\cap \mathrm{top}(\mathrm{Bs}(\psi))$,
       where $\mathrm{top}(\mathrm{Bs}(\psi))$ denotes the top component of the base locus of $\psi$;
  \item the rational map $\eta:\PP^5\dashrightarrow Z\subset\PP^{10}$ 
   defined by the octic hypersurfaces with triple points along $T$, and where 
   $Z\subset\PP^{10}$ is a $4$-dimensional linear section of $\mathbb{G}(1,5)\subset\PP^{14}$;
 \item the restriction of $\eta$ to the cubic fourfold $X$, which is a birational map onto $Z$.
\end{enumerate}
Now we can quickly get information on the maps, \emph{e.g.} on the inverse of the last one:
{\footnotesize
\begin{Verbatim}[commandchars=&\[\]]
&colore[darkorange][i13 :] g = &colore[airforceblue][last] oo;
o13 : RationalMap (birational map from hypersurface in PP^5 to 4-dimensional subvariety of PP^10)
&colore[darkorange][i14 :] &colore[airforceblue][describe inverse] g 
o14 = rational map defined by forms of degree 5
      source variety: 4-dimensional variety of degree 14 in PP^10 cut out by 15
                      hypersurfaces of degree 2
      target variety: smooth cubic hypersurface in PP^5
      birationality: true 
      projective degrees: {14, 70, 80, 24, 3} 
\end{Verbatim}
} \noindent 

\bibliographystyle{amsalpha}
\bibliography{bibliography}

\end{document}